\documentclass[11pt,reqno]{amsart}

\usepackage{amsfonts, amsthm, amsmath}
\allowdisplaybreaks[4]

\usepackage{rotating}

\usepackage{tikz}

\usepackage{graphics}

\usepackage{amssymb}

\usepackage{circledsteps}

\usepackage{amsmath}

\usepackage{amscd}

\usepackage[latin2]{inputenc}

\usepackage{t1enc}

\usepackage[mathscr]{eucal}

\usepackage{indentfirst}

\usepackage{graphicx}

\usepackage{graphics}

\usepackage{pict2e}

\usepackage{mathrsfs}

\usepackage{enumerate}

\usepackage[pagebackref]{hyperref}
\hypersetup{colorlinks=true}
\usepackage{cite}
\usepackage{color}
\usepackage{epic}
\usepackage{hyperref} %this gives clickable references, which is nice
\usepackage{framed}
\usepackage{mathabx}
\numberwithin{equation}{section}
\theoremstyle{plain}

\newtheorem{theorem}{Theorem}[section]

\newtheorem{lemma}[theorem]{Lemma}

\newtheorem{corollary}[theorem]{Corollary}

\newtheorem{proposition}[theorem]{Proposition}

\theoremstyle{definition}

\newtheorem{Def}[theorem]{Definition}

\newtheorem{example}[theorem]{Example}

\newtheorem{remark}[theorem]{Remark}

\newtheorem{?}[theorem]{Problem}

\newtheoremstyle{named}{}{}{\itshape}{}{\bfseries}{.}{.5em}{#1\thmnote{ #3}}
\theoremstyle{named}

\newcommand{\f}[1]{\ifthenelse{\equal{#1}{1}}{(q;q)_\infty}{(q^{#1};q^{#1})_{\infty}}}

\def\cD{\mathcal{D}}
\def\cP{\mathcal{P}}
\def\cA{\mathcal{A}}
\def\cB{\mathcal{B}}
\def\cC{\mathcal{C}}

\def\ocP{\overline{\mathcal{P}}}

\def\ri{\rightarrow}
\def\la{\lambda}

\def\op{\overline{p}}
\def\os{\overline{s}}
\def\og{\overline{g}}

\begin{document}
\title[The combinatorics of identities involving overpartitions with distinct parts]{The combinatorics of identities involving overpartitions with distinct parts}

\author[H. Li]{Haijun Li}
\address[Haijun Li]{College of Mathematics and Statistics, Chongqing University, Chongqing 401331, P.R. China}
\email{lihaijun@cqu.edu.cn}

\date{\today}

\begin{abstract}
Recently, Andrews and EI Bachraoui discovered several companions for some famous $q$-series formulas, and derived some new identities involving partitions and overpartitions with distinct parts. In this paper, we shall refine their results by the number of parts of partitions and furthermore, we will also provide the combinatorial proofs for those partition identities.
\end{abstract}


\maketitle

\section{Introduction}\label{sec:intro}

We begin with the notation of q-series. Let $q$ denote a complex number with $|q|<1$. Here and in what follows, we adopt the standard $q$-series notation \cite{GR90}. We let
\begin{align*}
&(a; q)_n=(1-a)(1-aq)\cdots (1-aq^{n-1})\text{ for }n\geq 1,\ (a; q)_0=1,\\
&(a; q)_{\infty}=\lim_{n\ri\infty}(a; q)_n,\text{ and }(a_1, ..., a_m; q)_n=(a_1; q)_n\cdots (a_m; q)_n.
\end{align*}

Now we introduce some necessary knowledge for the integer partition. A {\it partition} $\la$ of a positive integer $n$ is defined as a non-increasing sequence of positive integers $(\la_1, \la_2, ..., \la_{\ell})$, such that $\la_1+\la_2+\cdots +\la_{\ell}=n$. Write the weight $|\la|=n$, and the terms $\la_i$ are called the {\it parts} of $\la$, the number of parts of $\la$ is called the {\it length} of $\la$, denoted $\ell(\la)$. For more details about integer partitions see \cite{andtp}. Moreover, an overpartition \cite{CL04} of $n$ is a partition of $n$ where the first occurrence of each part may be overlined. The number of overpartitions of $n$, written $\op(n)$, has the following generating function
\begin{align*}
\sum_{n\geq 0}\op(n)q^n=\frac{(-q; q)_{\infty}}{(q; q)_{\infty}}.
\end{align*}
Note that overlined parts in overpartitions are distinct by definition. Then we say that an overpartition has distinct parts if its non-overlined parts are also distinct. So if we say that a non-overlined part $\la_k$ is equal to an overlined part $\la_m$, then only the numbers without the overlines are equal. Now let $\op_d(n)$ be the number of overpartitions of $n$ into distinct parts. For example, $\op_d(4)=9$ enumerates
\begin{align*}
4, \overline{4}, 3+1, \overline{3}+1, 3+\overline{1}, \overline{3}+\overline{1}, \overline{2}+2, 2+\overline{1}+1, \overline{2}+\overline{1}+1.
\end{align*}

Recently, Andrews and EI Bachraoui discovered several companions for some famous $q$-series formulas, and derived some new identities involving partitions and overpartitions. The main aim of this paper is to offer the bijective proofs for the following identities appearing in Andrews and EI Bachraoui's paper~\cite{AB25}. Firstly, let $\os(\pi)$ (resp. $\og(\pi)$) denote the smallest (resp. greatest) overlined part of the overpartition $\pi$.

\begin{theorem}[{cf. \cite[Def. 1 and Thm. 1]{AB25}}]\label{thm:1}
For any positive integer $n$, let $A(n)$ denote the number of overpartitions $\pi$ of $n$ into distinct parts with smallest overlined part such that the remaining overlined parts are even and $>2\os(\pi)$ and the non-overlined parts are odd and $<2\os(\pi)-1$. Furthermore, let $A_0(n)$ (resp. $A_1(n)$) denote the number of overpartitions counted by $A(n)$ wherein the number of parts is even (resp. odd) and let
\begin{align*}
A'(n)=A_1(n)-A_0(n).
\end{align*}
Then we have
\begin{align*}
(a)\quad &\sum_{n\geq 1}A(n)q^n:=\sum_{n\geq 1}q^n(-q^{2n+2}; q^2)_{\infty}(-q; q^2)_{n-1}=\frac{q}{1-q}(-q^2; q^2)_{\infty},\\
(b)\quad &\sum_{n\geq 1}A'(n)q^n:=\sum_{n\geq 1}q^n(q^{2n+2}; q^2)_{\infty}(q; q^2)_{n-1}=\frac{q}{1-q}(q^2; q^2)_{\infty}.
\end{align*}
\end{theorem}  

\begin{corollary}[{cf. \cite[Cor. 1]{AB25}}]\label{cor:1}
For any positive integer $n$, let $p_{ed}(n)$ be the number of partitions of $n$ into even not repeating parts and $p'_{ed}(n)$ be the number of partitions of $n$ counted by $p_{ed}(n)$ with an even number of parts minus those with an odd number of parts. Then we have
\begin{align*}
(a)\quad &A(n)-A(n-1)=p_{ed}(n-1),\\
(b)\quad &A'(n)-A'(n-1)=p'_{ed}(n-1).
\end{align*}
\end{corollary}

\begin{theorem}[{cf. \cite[Def. 2 and Thm. 2]{AB25}}]\label{thm:2}
For any positive integer $n$, let $B(n)$ denote the number of overpartitions $\pi$ of $n$ into distinct parts with smallest overlined part such that the remaining overlined parts are odd and $>2\os(\pi)+1$ and the non-overlined parts are even and $<2\os(\pi)$. Furthermore, let $B_0(n)$ (resp. $B_1(n)$) denote the number of overpartitions counted by $B(n)$ wherein the number of parts is even (resp. odd) and let
\begin{align*}
B'(n)=B_1(n)-B_0(n).
\end{align*}
Then we have
\begin{align*}
(a)\quad &\sum_{n\geq 1}B(n)q^n:=\sum_{n\geq 1}q^n(-q^{2n+3}; q^2)_{\infty}(-q^2; q^2)_{n-1}=\frac{q}{1-q^2}(-q; q^2)_{\infty},\\
(b)\quad &\sum_{n\geq 1}B'(n)q^n:=\sum_{n\geq 1}q^n(q^{2n+3}; q^2)_{\infty}(q^2; q^2)_{n-1}=\frac{q}{1-q}(q^3; q^2)_{\infty}.
\end{align*}
\end{theorem}  

\begin{corollary}[{cf. \cite[Cor. 2]{AB25}}]\label{cor:2}
For any positive integer $n>1$, let $p_{od}(n)$ be the number of partitions of $n$ into odd distinct parts and $p_{od>1}(n)$ be the number of partitions of $n$ into odd distinct parts $>1$. Then let $p'_{od>1}(n)$ be the number of partitions of $n$ counted by $p_{od>1}(n)$ whose number of parts is even minus the number of partitions counted by $p_{od>1}(n)$ wherein the number of parts is odd. Then we have
\begin{align*}
(a)\quad &B(n)-B(n-2)=p_{od}(n-1),\\
(b)\quad &B'(n)-B'(n-1)=p'_{od>1}(n-1).
\end{align*}
\end{corollary}

\begin{theorem}[{cf. \cite[Def. 3 and Thm. 3]{AB25}}]\label{thm:3}
For any positive integer $n$, let $C(n)$ denote the number of overpartitions $\pi$ of $n$ into distinct parts where $\og(\pi)$ occurs only overlined and the non-overlined parts which are $<\og(\pi)$ and must be even $>2\og(\pi)$. Furthermore, let $C_0(n)$ (resp. $C_1(n)$) denote the number of overpartitions counted by $C(n)$ wherein the number of parts $<\og(\pi)$ is even (resp. odd) and let
\begin{align*}
C'(n)=C_0(n)-C_1(n).
\end{align*}
Then we have
\begin{align*}
(a)\quad &\sum_{n\geq 1}C(n)q^n:=\sum_{n\geq 1}q^n(-q^{2n+2}; q^2)_{\infty}(-q; q)^2_{n-1}=\frac{1}{2}((-q; q)^2_{\infty}-(-q^2; q^2)_{\infty}),\\
(b)\quad &\sum_{n\geq 1}C'(n)q^n:=\sum_{n\geq 1}q^n(-q^{2n+2}; q^2)_{\infty}(q; q)^2_{n-1}=\frac{1}{2}((-q^2; q^2)_{\infty}-(q; q)^2_{\infty}).
\end{align*}
\end{theorem}

\begin{corollary}[{cf. \cite[Cor. 3]{AB25}}]\label{cor:3}
For any positive integer $n>1$, let $\op_{no}(n)$ be the number of overpartitions of $n$ wherein the non-overlined parts are odd and $\op'_d(n)$ be the number of overpartitions of $n$ into distinct parts wherein the number of parts is even minus the number of those wherein the number of parts id odd. Then we have
\begin{align*}
(a)\quad &C(n)=\frac{\op_{no}(n)-p_{ed}(n)}{2},\\
(b)\quad &C'(n)=\frac{p_{ed}(n)-\op'_d(n)}{2}.
\end{align*}
\end{corollary}

\begin{theorem}[{cf. \cite[Def. 4 and Thm. 4]{AB25}}]\label{thm:4}
For any positive integer $n$, let $D(n)$ denote the number of overpartitions $\pi$ of $n$ into distinct parts with greatest part overlined such that the non-overlined parts which are $\leq \og(\pi)$ and must be odd $>2\og(\pi)+1$. Furthermore, let $D_0(n)$ (resp. $D_1(n)$) denote the number of overpartitions counted by $D(n)$ wherein the number of parts $\leq \og(\pi)$ is even (resp. odd) and let
\begin{align*}
D'(n)=D_1(n)-D_0(n).
\end{align*}
Then we have
\begin{align*}
(a)\quad &\sum_{n\geq 1}D(n)q^n:=\sum_{n\geq 1}q^n(1+q^n)(-q^{2n+3}; q^2)_{\infty}(-q; q)^2_{n-1}=\frac{1}{1+q}((-q; q)^2_{\infty}-(-q; q^2)_{\infty}),\\
(b)\quad &\sum_{n\geq 1}D'(n)q^n:=\sum_{n\geq 1}q^n(1-q^n)(-q^{2n+3}; q^2)_{\infty}(q; q)^2_{n-1}=\frac{1}{1+q}((1-q)(-q^3; q)_{\infty}-(q; q)^2_{\infty}).
\end{align*}
\end{theorem}  

\begin{corollary}[{cf. \cite[Cor. 4]{AB25}}]\label{cor:4}
For any positive integer $n>1$ we have
\begin{align*}
(a)\quad &D(n)+D(n-1)=\op_{no}(n)-p_{od}(n),\\
(b)\quad &D'(n)+D'(n-1)=p_{od>1}(n)-p_{od>1}(n-1)-\op'_{d}(n).
\end{align*}
\end{corollary}

In the remaining part of this section, we shall introduce our main results. Firstly, we will refine the Theorems \ref{thm:1} and \ref{thm:2} with the number of parts except for the smallest overlined part. The following theorems will express the results of refinements. 

\begin{theorem}\label{thm:res1}
For any integers $n\geq 1$ and $m\geq 0$, let $A(n, m)$ be the number of overpartitions counted by $A(n)$ with $m$ parts except for the smallest overlined part. Then we have
\begin{align}
\sum_{m\geq 0, n\geq 1}A(n, m)x^mq^n:=\sum_{n\geq 1}q^n(-xq^{2n+2}; q^2)_{\infty}(-xq; q^2)_{n-1}=\frac{q}{1-q}(-xq^2; q^2)_{\infty}.\label{id:res1}
\end{align}
\end{theorem}

\begin{theorem}\label{thm:res2}
For any integers $n\geq 1$ and $m\geq 0$, let $p_{ed}(n, m)$ be the number of partitions counted by $p_{ed}(n)$ with $m$ parts. Then we have
\begin{align}
A(n, m)-A(n-1, m)=p_{ed}(n-1, m).\label{id:res2}
\end{align}
\end{theorem}

\begin{theorem}\label{thm:res3}
For any integers $n\geq 1$ and $m\geq 0$, let $B(n, m)$ be the number of overpartitions counted by $B(n)$ with $m$ parts except for the smallest overlined part. Then we have
\begin{align}
\sum_{m\geq 0, n\geq 1}B(n, m)x^mq^n:=\sum_{n\geq 1}q^n(-xq^{2n+3}; q^2)_{\infty}(-xq^2; q^2)_{n-1}=\frac{q}{1-q}(-xq^3; q^2)_{\infty}.\label{id:res3}
\end{align}
\end{theorem}

\begin{theorem}\label{thm:res4}
For any integers $n\geq 1$ and $m\geq 0$, let $p_{od>1}(n, m)$ be the number of partitions counted by $p_{od>1}(n)$ with $m$ parts. Then we have
\begin{align}
B(n, m)-B(n-1, m)=p_{od>1}(n-1, m).\label{id:res4}
\end{align}
\end{theorem}

Then we can obtain the combinatorial proofs of Corollaries \ref{cor:3} and \ref{cor:4} by a series of bijections which will be written in Section \ref{sec:pf_c34}. In fact, we also give two more accurate results about those two corollaries by parity of number of parts. Let $\op^e_d(n)$ (resp. $\op^o_d(n)$) denote the number of overpartitions of $n$ into distinct parts where the number of parts is even (resp. odd). 

\begin{theorem}\label{thm:res5}
For any integer $n\geq 0$, we have
\begin{align}
C_0(n)&=\frac{\op_d^o(n)}{2},\label{id:res5-1}\\
C_1(n)&=\frac{\op_d^e(n)-p_{ed}(n)}{2}.\label{id:res5-2}
\end{align}
\end{theorem}

\begin{theorem}\label{thm:res6}
For any integer $n\geq 0$, we have
\begin{align}
D_0(n)+D_0(n-1)&=\op_d^e(n)-p_{od>1}(n),\label{id:res6-1}\\
D_1(n)+D_1(n-1)&=\op_d^o(n)-p_{od>1}(n-1).\label{id:res6-2}
\end{align}
\end{theorem}

We organize the rest of the paper as follows. The proofs of Theorems \ref{thm:res1}-\ref{thm:res4} will be presented in Section \ref{sec:pf_c12}. Furthermore, we will also give the bijective proofs for Theorems \ref{thm:res2} and \ref{thm:res4} which imply the combinatorial proofs of first two corollaries. Next in Section \ref{sec:pf_c34}, we will prove the Corollaries \ref{cor:3} and \ref{cor:4} by a series of bijections.

%%%%%%%%%%%%%%%%%%%%%%%%%%%%%%%%%%%%%%%%%%%%%%%%%%
\section{The Proofs of Theorems \ref{thm:res1}-\ref{thm:res4}}\label{sec:pf_c12}

In this section, we will divide it into two parts. We shall provide the analytic proofs of Theorems \ref{thm:res1}-\ref{thm:res4} in the first part. And in the second part, the combinatorial proofs of Theorems \ref{thm:res2} and \ref{thm:res4} will be presented. Before stating the analytic proofs, we introduce the $q$-Gauss sum~\cite{GR90} as follows:
\begin{align}
_2\phi_1(a, b; c; q, c/ab)=\sum_{n\geq 0}\frac{(a; q)_n(b; q)_n}{(c; q)_n}\frac{(c/ab)^n}{(q; q)_n}=\frac{(c/a, c/b; q)_{\infty}}{(c, c/ab; q)_{\infty}}.\label{id:qGauss}
\end{align}
Then let $(q, a, b, c)\ri (q^2, q^2, -xq, -xq^4)$ and $(q, a, b, c)\ri (q^2, q^2, -xq^2, -xq^5)$ in \eqref{id:qGauss} respectively, we have
\begin{align}
\sum_{n\geq 0}\frac{(-xq; q^2)_nq^n}{(-xq^2; q^2)_{n+1}}&=\frac{1}{1-q},\label{id:tool1}\\
\sum_{n\geq 0}\frac{(-xq^2; q^2)_nq^n}{(-xq^5; q^2)_n}&=\frac{1+xq^3}{1-q}.\label{id:tool2}
\end{align}

\begin{proof}[Proof of Theorem \ref{thm:res1}]
We have
\begin{align*}
\sum_{n\geq 1}q^n(-xq^{2n+2}; q^2)_{\infty}(-xq; q^2)_{n-1}&=\sum_{n\geq 0}q^{n+1}(-xq^{2n+4}; q^2)_{\infty}(-xq; q^2)_{n}\\
&=q(-xq^2; q^2)_{\infty}\sum_{n\geq 0}\frac{(-xq; q^2)_n}{(-xq^2; q^2)_{n+1}}q^n\\
&=\frac{q}{1-q}(-xq^2; q^2)_{\infty}.
\end{align*}
Use the equation \eqref{id:tool1} in the last step.
\end{proof}

\begin{proof}[Analytic proof of Theorem \ref{thm:res2}]
Firstly we define $A(0, m)=0$ for any integer $m\geq 0$ since the smallest overlined part must exist by the definition of $A(n)$. Then multiply the left side of the equation \eqref{id:res2} by $x^mq^{n-1}$ and sum over all $n\geq 1$ and $m\geq 0$, we have
\begin{align*}
\sum_{m\geq 0, n\geq 1}(A(n, m)-A(n-1, m))x^mq^{n-1}&=\frac{1}{q}\sum_{m\geq 0, n\geq 1}A(n, m)x^mq^n-\sum_{m\geq 0, n\geq 1}A(n-1, m)x^mq^{n-1}\\
&=\frac{1}{q}\cdot \frac{q}{1-q}\cdot(-xq^2; q^2)_{\infty}-\frac{q}{1-q}\cdot(-xq^2; q^2)_{\infty}\\
&=(-xq^2; q^2)_{\infty}.
\end{align*}
On the other hand, we easily know that
\begin{align*}
\sum_{m, n\geq 0}p_{ed}(n, m)x^mq^{n}=(-xq^2; q^2)_{\infty}.
\end{align*}
Hence, the equation \eqref{id:res2} holds.
\end{proof}

\begin{proof}[Proof of Theorem \ref{thm:res3}]
We have
\begin{align*}
\sum_{n\geq 1}q^n(-xq^{2n+3}; q^2)_{\infty}(-xq^2; q^2)_{n-1}&=\sum_{n\geq 0}q^{n+1}(-xq^{2n+5}; q^2)_{\infty}(-xq^2; q^2)_n\\
&=q(-xq^5; q^2)_{\infty}\sum_{n\geq 0}\frac{(-xq^2; q^2)_{n}}{(-xq^5; q^2)_{n}}q^n\\
&=\frac{q}{1-q}(-xq^3; q^2)_{\infty}
\end{align*}
Use the equation \eqref{id:tool2} in the last step.
\end{proof}

\begin{proof}[Analytic proof of Theorem \ref{thm:res4}]
Firstly we define $B(0, m)=0$ for any integer $m\geq 0$ since the smallest overlined part must exist by the definition of $B(n)$. Then multiply the left side of the equation \eqref{id:res2} by $x^mq^{n-1}$ and sum over all $n\geq 1$ and $m\geq 0$, we have
\begin{align*}
\sum_{m\geq 0, n\geq 1}(B(n, m)-B(n-1, m))x^mq^{n-1}&=\frac{1}{q}\sum_{m\geq 0, n\geq 1}B(n, m)x^mq^n-\sum_{m\geq 0, n\geq 1}B(n-1, m)x^mq^{n-1}\\
&=\frac{1}{q}\cdot \frac{q}{1-q}\cdot(-xq^3; q^2)_{\infty}-\frac{q}{1-q}\cdot(-xq^3; q^2)_{\infty}\\
&=(-xq^3; q^2)_{\infty}.
\end{align*}
On the other hand, we know that
\begin{align*}
\sum_{m, n\geq 0}p_{od>1}(n, m)x^mq^{n}=(-xq^3; q^2)_{\infty}.
\end{align*}
Hence, the equation \eqref{id:res4} holds.
\end{proof}

The next part will give the combinatorial proofs for Theorem \ref{thm:res2} and \ref{thm:res4}. At the first, we review some basic concepts. Let $\cA(n, m)$ be the set of overpartitions counted by $A(n, m)$ and $\cP_{ed}(n, m)$ be the set of partitions counted by $p_{ed}(n, m)$. Recall that $\os(\pi)$ is the smallest overlined part of an overpartition $\pi$. Specially, we define $\os(\la)=0$ if $\la\in  \cP_{ed}(n, m)$. Then we obtain the following lemma.

\begin{lemma}\label{lem:main1}
For any integers $n\geq 1$ and $m\geq 0$, there exists a bijection
\begin{align*}
\varphi:\cA(n, m)&\ri\cA(n-1, m)\cup \cP_{ed}(n-1, m)\\
\pi&\mapsto\la  
\end{align*}
such that $|\pi|=|\la|+1$, $\ell(\pi)=\ell(\la)$ if $\la\in \cA(n-1, m)$ and $\ell(\pi)=\ell(\la)+1$ if $\la\in \cP_{ed}(n-1, m)$. Consequently, Theorem \ref{thm:res2} holds true.
\end{lemma}

\begin{proof}
Firstly note that if $n=1$ then $\cA(1, 0)=\{\overline{1}\}$ corresponds to $\cP_{ed}(0, 0)=\{\epsilon\}$ where $\epsilon$ stands for the empty partition since $|\cA(0, m)|=A(0, m)=0$ for all $m\geq 0$. Now for any given $\pi\in \cA(n, m)$ and its smallest overlined part $\os(\pi)$, we shall describe the map $\varphi$ as follows. Firstly we obtain $\pi'$ by setting $\os(\pi')=\os(\pi)-1$ in $\pi$ and letting remaining parts of $\pi$ stay the same. Therefore, now there are three following cases. 
\begin{description}

\item[CASE I] if $\os(\pi)=1$, then there are no non-overlined odd parts and overlined even parts $\geq 4$ in $\pi$. Now we remove the overlines of all parts of $\pi'$ to obtain $\la\in \cP_{ed}(n-1, m)$ and the smallest part of $\la$ $\geq 4$ (that is, $\la$ does not contain the part $2$).

\item[CASE II] if $\os(\pi)>1$ and there is no non-overlined odd part equal to $2\os(\pi)-3$ in $\pi$, then easily check that the overpartition $\pi'=\la$ belongs to the set $\cA(n-1, m)$ and there is no overlined even part equal to $2\os(\la)+2$. This is because after setting $\os(\pi')=\os(\pi)-1$ the statement ``there is no non-overlined odd part equal to $2\os(\pi)-3$'' means that all non-overlined odd parts $\leq 2\os(\pi)-5=2\os(\pi')-3$ in $\pi'$. At this moment, $\pi'$ satisfies all conditions in the definition of $\cA(n-1, m)$.

\item[CASE III] if $\os(\pi)>1$ and there are certain non-overlined odd parts
\begin{align*}
\pi_k=2\os(\pi)-3,..., \pi_{k+p-1}=2\os(\pi)-2p-1, \pi_{k+p}\leq 2\os(\pi)-2p-5
\end{align*}
in $\pi$, then we can obtain $\pi''$ by the following three steps: (1). set $\os(\pi'')=\os(\pi')-p=\os(\pi)-p-1$; (2). change $\pi_k, ..., \pi_{k+p-1}$ to the overlined even parts $2\os(\pi)-2, ..., 2\os(\pi)-2p$ by adding one; and (3). the remaining parts stay the same as those in $\pi'$. Next we shall divide it into two following subcases to discuss.
\begin{description}
\item[CASE III-1] if $\os(\pi'')\geq 1$, then we obtain $\pi''=\la\in\cA(n-1, m)$. Since the smallest overlined even part except for $\os(\pi'')$ equals to $2\os(\pi)-2p=2(\os(\pi'')+p+1)-2p=2\os(\pi'')+2$. And the largest non-overlined odd part $\leq2\os(\pi)-2p-5=2(\os(\pi'')+p+1)-2p-5=2\os(\pi'')-3$. Hence, $\pi''=\la$ satisfies all conditions in the definition of $\cA(n-1, m)$.

\item[CASE III-2] if $\os(\pi'')=0$, then there must be no a non-overlined odd part in $\pi''$. Hence, we can remove the overlines of all overlined even parts in $\pi''$ to obtain $\la\in \cP_{ed}(n-1, m)$ and $\la$ must contain the part $2$. Since $\os(\pi)=p+1$ at this time, we see that the smallest even part must be $2\os(\pi)-2p=2(p+1)-2p=2$.
\end{description}
\end{description}

We have explained the validity of the resulting overpartitions $\la$ obtained by above operations in each case. Now we know that these overpartitions $\la$ are well-defined. Furthermore, we can also easily see that these resulting overpartitions $\la$ are not repeated. Firstly, the resulting partitions $\la$ in $\cP_{ed}(n-1, m)$ are not repeated since we obtain two classes: (1). the partitions which do not contain the part $2$ and (2). the partitions which must contain the part $2$. Next, we see that the resulting overpartitions $\la$ in $\cA(n-1, m)$ are also not repeated. In fact, there are two classes overpartitions after carrying out this map $\varphi$: (1). the overpartitions which do not contain the part $2\os(\la)-2$ and (2). the overpartitions which must contain the part $2\os(\la)-2$.

On the other hand, we will give the statements of inverse map $\varphi^{-1}$ step by step. Firstly, given $\la\in \cA(n-1, m)\cup\cP_{ed}(n-1, m)$, then there are three cases as follows.
\begin{description}
\item[CASE I'] if $\la\in \cP_{ed}(n-1, m)$ and the smallest part of $\la\geq 4$, then we can add the overlines above each part of $\la$ and insert $\overline{1}$ as a new smallest part into $\la$ to obtain $\pi\in \cA(n, m)$. At this time, we know $\os(\pi)=1$, and there are no non-overlined odd parts in $\pi$.

\item[CASE II'] if $\la\in \cA(n-1, m)$ and there is no overlined even part equal to $2\os(\la)+2$ in $\la$, then we can obtain $\pi$ by setting $\os(\pi)=\os(\la)+1$. We claim that $\pi\in \cA(n, m)$ since the statement ``there is no overlined even part equal to $2\os(\la)+2$ in $\la$'' means that all overlined even parts $\geq 2\os(\la)+4=2\os(\pi)+2$. Hence, the resulting overpartition $\pi$ satisfies all conditions in the definition of $\cA(n, m)$.

\item[CASE III'] we divide the remaining classes of overpartitions into two cases as follows:
\begin{description}
\item[CASE III'-1] if $\la\in \cA(n-1, m)$ and there are certain overlined even parts
\begin{align*}
\la_{s}\geq 2\os(\la)+2t+4, \la_{s+1}=2\os(\la)+2t, ... , \la_{s+t}=2\os(\la)+2  
\end{align*}
in $\la$, then we can obtain $\la'$ by setting $\os(\la')=\os(\la)+1$ and the remaining parts stay the same as $\la$.

\item[CASE III'-2] if $\la\in \cP_{ed}(n-1, m)$ and the smallest part of $\la$ equals to $2$, then we can add the overlines above each part of $\la$ and insert $\overline{1}$ as a new smallest part into $\la$ to obtain $\la'$. Note that $\os(\la)=0$ and $\os(\la')=1$ and $\la_{s+t}=2$ must exist. Hence, we can consider $\la'$ in this subcase as same as $\la'$ in above subcase.
\end{description}
Therefore, we can describe a common operation in $\la'$ of above two subcases. We can obtain $\pi$ from $\la'$ by the following three steps: (1). set $\os(\pi)=\os(\la')+t=\os(\la)+t+1$; (2). change $\la_{s+1}, ..., \la_{s+t}$ to the non-overlined odd parts $2\os(\la)+2t-1, ..., 2\os(\la)+1$; and (3). the remaining parts stay the same as those in $\la'$. 

We claim that $\pi\in \cA(n, m)$. Firstly the largest non-overlined odd parts $= 2\os(\la)+2t-1=2(\os(\pi)-t-1)+2t-1=2\os(\pi)-3$. And the smallest overlined even part except for $\os(\pi)$ $\geq 2\os(\la)+2t+4=2(\os(\pi)-t-1)+2t+4=2\os(\pi)+2$. Hence, the resulting overpartition $\pi$ satisfies all conditions in definition of $\cA(n, m)$.
\end{description}

The same discussions with above map tell us that the map $\varphi^{-1}$ is also well-defined and $\varphi^{-1}$ is exactly an inverse map step by step. Hence, we can claim the map $\varphi$ is a bijection. Moreover, if $\la=\varphi(\pi)$ for any given $\pi\in\cA(n, m)$, then we have $|\pi|=|\la|+1$, $\ell(\pi)=\ell(\la)$ if $\la\in \cA(n-1, m)$ and $\ell(\pi)=\ell(\la)+1$ if $\la\in \cP_{ed}(n-1, m)$ from the description of the bijection $\varphi$.
\end{proof}

\begin{example}
For $n=15$, we have $A(15)-A(14)=19-14=5=p_{ed}(14)$. We have the following correspondences case by case.
\begin{align*}
\begin{array}{ccccccccc}
\textbf{CASE I:} & \overline{14}+\overline{1}\ri 14; & \overline{10}+\overline{4}+\overline{1}\ri 10+4; & \overline{8}+\overline{6}+\overline{1}\ri 8+6;\\
\textbf{CASE II:} & \overline{15}\ri \overline{14}; & \overline{14}+1\ri \overline{13}+1; & \overline{12}+3\ri \overline{11}+3;\\
&\overline{11}+3+1\ri \overline{10}+3+1; & \overline{10}+5\ri \overline{9}+5 ;& \overline{9}+5+1\ri \overline{8}+5+1 ;\\
& \overline{8}+7\ri \overline{7}+7 ;& \overline{7}+7+1\ri 7+\overline{6} +1; & \overline{7}+5+3\ri \overline{6}+5+3;\\
&  \overline{6}+5+3+1\ri \overline{5}+5+3+1 ;& \overline{10}+\overline{4}+1\ri \overline{10}+\overline{3}+1 ;& \overline{12}+\overline{3}\ri \overline{12}+\overline{2};\\
\textbf{CASE III-1:} & 9+\overline{6}\ri \overline{10}+\overline{4}; & 7+\overline{5}+3\ri \overline{8}+\overline{3}+3; & \\
\textbf{CASE III-2:} & \overline{8}+\overline{3}+3+1\ri 8+4+2; &\overline{12}+\overline{2}+1\ri 12+2. &
\end{array}
\end{align*}
\end{example}

\begin{remark}
For any integers $m\geq 0$ and $n\geq 1$, let $p_{ed}^{e}(n)$ (resp. $p_{ed}^o(n)$) be the number of partitions of $n$ counted by $p_{ed}(n)$ where the number of parts is even (resp. odd). Then we have
\begin{align*}
p'_{ed}(n)=p_{ed}^e(n)-p_{ed}^o(n)=\sum_{m\text{ even}}p_{ed}(n, m)-\sum_{m\text{ odd}}p_{ed}(n, m).
\end{align*} 
In fact, note that we prove Theorem \ref{thm:res2} is a refinement of Corollary \ref{cor:1}. Firstly, by our bijection $\varphi$ we can easily see that
\begin{align*}
A(n)-A(n-1)&=\sum_{m\geq 0}A(n, m)-\sum_{m\geq 0}A(n-1, m)=\sum_{m\geq 0}(A(n, m)-A(n-1, m))\\
&=\sum_{m\geq 0}p_{ed}(n-1, m)=p_{ed}(n-1).
\end{align*}   
On the other hand, we have
\begin{align*}
A'(n)-A'(n-1)&=(A_1(n)-A_1(n-1))-(A_0(n)-A_0(n-1))\\
&=\sum_{m\text{ even}}(A(n, m)-A(n-1, m))-\sum_{m\text{ odd}}(A(n, m)-A(n-1, m))\\
&=\sum_{m\text{ even}}p_{ed}(n-1, m)-\sum_{m\text{ odd}}p_{ed}(n-1, m)\\
&=p_{ed}'(n-1).
\end{align*}
Indeed $m$ in all above equations is finite with a given $n$.
\end{remark}

%%%%%%%%%%%%%%%%%%%%%%%%%%%%%%%%%%%
%第二个定理的双射

At the end of this section, we will have a similar discussion on the Theorem \ref{thm:res4}. Let $\cB(n, m)$ be the set of overpartitions counted by $B(n, m)$ and $\cP_{od>1}(n, m)$ be the set of partitions counted by $p_{od>1}(n, m)$. Then we can obtain the following lemma.

\begin{lemma}\label{lem:main2}
For any integers $n\geq 1$ and $m\geq 0$, there exists a bijection
\begin{align*}
\psi:\cB(n, m)&\ri\cB(n-1, m)\cup \cP_{od>1}(n-1, m)\\
\pi&\mapsto\la  
\end{align*}
such that $|\pi|=|\la|+1$, $\ell(\pi)=\ell(\la)$ if $\la\in \cB(n-1, m)$ and $\ell(\pi)=\ell(\la)+1$ if $\la\in \cP_{od>1}(n-1, m)$. Consequently, Theorem \ref{thm:res4} holds true.
\end{lemma}

\begin{proof}
The proof is similar to Lemma \ref{lem:main1}, so we will omit some details here and only provide the specific operations. Firstly we describe the map $\psi$ as follows. For any given $\pi\in \cB(n, m)$,  we obtain $\pi'$ by setting $\os(\pi')=\os(\pi)-1$ in $\pi$ and letting remaining parts of $\pi$ stay the same. Then there are following three cases.
\begin{description}
\item[CASE I] if $\os(\pi)=1$, then we remove the overlines of all parts of $\pi'$ to obtain $\la\in \cP_{od>1}(n-1, m)$ and the smallest part $\geq 5$. 

\item[CASE II] if $\os(\pi)>1$ and there is no non-overlined even part equal to $2\os(\pi)-2$ in $\pi$, then $\pi'=\la\in \cB(n-1, m)$ where there is no overlined odd part equal to $2\os(\la)+3$.

\item[CASE III] if $\os(\pi)>1$ and there are non-overlined even parts
\begin{align*}
\pi_k=2\os(\pi)-2, ..., \pi_{k+p-1}=2\os(\pi)-2p, \pi_{k+p}\leq 2\os(\pi)-2p-4
\end{align*}
in $\pi$, then we can obtain $\pi''$ by the following three steps: (1). set $\os(\pi'')=\os(\pi')-p$; (2). change $\pi_k, ..., \pi_{k+p-1}$ to the overlined odd parts $2\os(\pi)-1, ..., 2\os(\pi)-2p+1$; and (3). the remaining parts stay the same as those in $\pi'$. Now there are two subcases as follows.
\begin{description}
\item[CASE III-1] if $\os(\pi'')\geq 1$, then we have $\pi''=\la\in \cB(n-1, m)$.

\item[CASE III-2] if $\os(\pi'')=0$, then we remove the overlines of all overlined odd parts in $\pi''$ to obtain $\la\in \cP_{od>1}(n-1, m)$ and $\la$ must contain the part $3$.
\end{description}
\end{description}

On the other hand, we can state the inverse map $\psi^{-1}$ which corresponds to $\psi$ step by step. For any given $\la\in \cB(n-1, m)\cup\cP_{od>1}(n-1, m)$, there are three following cases.
\begin{description}
\item[CASE I'] if $\la\in \cP_{od>1}(n-1, m)$ and the smallest part of $\la \geq 5$, then we can add the overlines above each part of $\la$ and insert $\overline{1}$ as a new smallest part into $\la$ to obtain $\pi\in \cB(n, m)$.

\item[CASE II'] if $\la \in \cB(n-1, m)$ and there is no overlined odd part equal to $2\os(\la)+3$ in $\la$, then we can obtain $\pi$ by setting $\os(\pi)=\os(\la)+1$.  

\item[CASE III'] there are two subcases as follows.
\begin{description}
\item[CASE III'-1] if $\la\in \cB(n-1, m)$ and there are overlined odd parts
\begin{align*}
\la_s\geq 2\os(\la)+2t+5, \la_{s+1}=2\os(\la)+2t+1, ..., \la_{s+t}=2\os(\la)+3
\end{align*}
in $\la$, then we can obtain $\la'$ by setting $\os(\la')=\os(\la)+1$ and the remaining parts stay the same as $\la$.

\item[CASE III'-2] if $\la\in \cP_{od>1}(n-1, m)$ and the smallest part of $\la$ equals to $3$. then we can add the overlines above each part of $\la$ and insert $\overline{1}$ as a new smallest part into $\la$ to obtain $\la'$.
\end{description}
Hence, we can obtain $\pi$ from $\la'$ by the following three steps: (1). set $\os(\pi)=\os(\la')+t$; (2). change $\la_{s+1}, ..., \la_{s+t}$ to the non-overlined even parts $2\os(\la)+2t, ..., \la_{s+t}=2\os(\la)+2$; and (3). the remaining parts stay the same as those in $\la'$.
\end{description}

Finally we can easily check that
\begin{itemize}
\item The resulting overpartitions in all above cases are well-defined and not repeated;

\item The map $\psi$ is a bijection such that $|\pi|=|\la|+1$, $\ell(\pi)=\ell(\la)$ if $\la\in \cB(n-1, m)$ and $\ell(\pi)=\ell(\la)+1$ if $\la\in \cP_{od>1}(n-1, m)$.
\end{itemize}
\end{proof}

\begin{example}
For $n=17$, we have $B(17)-B(16)=18-15=3=p_{od>1}(16)$. We have the following correspondences case by case.
\begin{align*}
\begin{array}{ccccccccc}
\textbf{CASE I:} & \overline{11}+\overline{5}+\overline{1}\ri 11+5; & \overline{9}+\overline{7}+\overline{1}\ri 9+7; & \\
\textbf{CASE II:} & \overline{17}\ri \overline{16}; & \overline{15}+2\ri \overline{14}+2; & \overline{13}+4\ri \overline{12}+4;\\
&\overline{11}+6\ri \overline{10}+6; & \overline{11}+4+2\ri \overline{10}+4+2 ;& \overline{9}+8\ri \overline{8}+8 ;\\
& \overline{9}+6+2\ri \overline{8}+6+2 ;& 10+\overline{7}\ri 10+\overline{6}; & 8+\overline{7}+2\ri 8+\overline{6}+2;\\
&  \overline{7}+6+4\ri \overline{6}+6+4 ;& 6+\overline{5}+4+2\ri 6+\overline{4}+4+2 ;& \overline{11}+\overline{4}+2\ri \overline{11}+\overline{3}+2;\\
&\overline{13}+\overline{4}\ri \overline{13}+\overline{3}; & \overline{15}+\overline{2}\ri \overline{15}+\overline{1};&\\
\textbf{CASE III-1:} & 8+\overline{5}+4\ri \overline{9}+4+\overline{3}; & & \\
\textbf{CASE III-2:} & \overline{13}+\overline{2}+2\ri 13+3. &  &
\end{array}
\end{align*}
\end{example}

\begin{remark}
For any integers $n\geq 1$ and $m\geq 0$, in fact Theorem \ref{thm:res4} tells us that 
\begin{align*}
B(n)-B(n-1)=p_{od>1}(n-1).
\end{align*}
Hence, we have 
\begin{align*}
B(n)-B(n-2)&=(B(n)-B(n-1))+(B(n-1)-B(n-2))\\
&=p_{od>1}(n-1)+p_{od>1}(n-2)\\
&=p_{od}(n-1).
\end{align*}
Moreover, let $p_{od>1}^e(n)$ (resp. $p_{od>1}^o(n)$) be the number of partitions of $n$ counted by $p_{od>1}(n)$ where the number of parts is even (resp. odd). Then by above bijection $\psi$ we have
\begin{align*}
B'(n)-B'(n-1)&=(B_1(n)-B_1(n-1))-(B_0(n)-B_0(n-1))\\
&=\sum_{m\text{ even}}(B(n, m)-B(n-1, m))-\sum_{m\text{ odd}}(B(n, m)-B(n-1, m))\\
&=\sum_{m\text{ even}}p_{od>1}(n-1, m)-\sum_{m\text{ odd}}p_{od>1}(n-1, m)\\
&=p_{od>1}^e(n-1)-p_{od>1}^o(n-1)=p'_{od>1}(n-1).
\end{align*}
\end{remark}

%%%%%%%%%%%%%%%%%%%%%%%%%%%%%%%%%%%%%%%%%%%%%%%%%%%
\section{The Proofs of Theorems \ref{thm:res5} and \ref{thm:res6}}\label{sec:pf_c34}

In this section, we shall give the combinatorial proofs of Corollaries \ref{cor:3} and \ref{cor:4} by proving Theorems \ref{thm:res5} and \ref{thm:res6}. Firstly recall that $\op_d(n)$ is the number of overpartitions of $n$ into distinct parts. Therefore, we have the following proposition.
\begin{proposition}
For any integer $n\geq 0$ we have
\begin{align}
\op_d(n)=\op_{no}(n).
\end{align}
\end{proposition}

\begin{proof}
Using Euler's partition theorem, we have
\begin{align*}
\sum_{n\geq 0}\op_d(n)q^n=(-q; q)^2_{\infty}=\frac{(-q; q)_{\infty}}{(q; q^2)_{\infty}}=\sum_{n\geq 0}\op_{no}(n)q^n.
\end{align*}
\end{proof}

Before proving Theorems \ref{thm:res5} and \ref{thm:res6}, we need some notations and definitions to help us explain these bijections more clearly.

\begin{Def}
For any integer $n\geq 0$ and given overpartition $\la$ counted by $\op_d(n)$, here are four types about its parts.
\begin{itemize}
\item[(i)] We say that the adjacent parts $\la_k$ and $\la_{k+1}$ form a pair of type I if $\la_k$ is overlined and $\la_{k+1}$ is non-overlined and $\la_k=\la_{k+1}$.

\item[(ii)] We say that the adjacent parts $\la_k$ and $\la_{k+1}$ form a pair of type II if $\la_k$ is non-overlined and $\la_{k+1}$ is overlined and $\la_{k}=\la_{k+1}+1$.

\item[(iii)] We say that the part $\la_k$ is a singleton of type I if $\la_k$ is overlined and there is no non-overlined part $\la_{k+1}=\la_k$ or if $\la_{k}$ is non-overlined and there is no overlined part $\la_{k-1}=\la_{k}$.

\item[(iv)] We say that the part $\la_k$ is a singleton of type II if $\la_k$ is overlined and there is no non-overlined part $\la_{k-1}=\la_k+1$ or if $\la$ is non-overlined and there is no overlined part $\la_{k+1}=\la_{k}-1$.
\end{itemize}
For example, given an overpartition $\la=\overline{12}+12+\overline{11}+11+\overline{9}+\overline{8}+8+7+\overline{6}+\overline{3}+3$, then we have
(i). the pairs of type I: $\overline{12}+12, \overline{11}+11, \overline{8}+8, \overline{3}+3$; (ii). the pairs of type II: $12+\overline{11}, 7+\overline{6}$; (iii). the singletons of type I: $\overline{9}, 7, \overline{6}$; and (iv). the singletons of type II: $\overline{12}, 11, \overline{9}, \overline{8}, 8, \overline{3}, 3$.
\end{Def}

%let $\ocP^e_d(n)$ (resp. $\ocP^o_d(n)$) denote the set of overpartitions of $n$ counted by $\op^e_d(n)$ (resp. $\op^o_d(n)$) and  

Moreover, let $\cC_0(n)$ (resp. $\cC_1(n), \cD_0(n), \cD_1(n)$) be the set of overpartitions of $n$ counted by $C_0(n)$ (resp. $C_1(n), D_0(n), D_1(n)$). Now we can pay attention to the proofs of above two theorems. 

\begin{proof}[Proof of Theorem \ref{thm:res5}]
We divide the proof into two parts.
 
\noindent\textbf{PART I}. We prove the identity \eqref{id:res5-1}. Firstly let $\ocP_d^{o1}(n)$ (resp. $\ocP_d^{o2}(n)$) be the set of overpartitions of $n$ counted by $\op_n^o(n)$ where the largest singleton of type I is overlined (resp. non-overlined). Since the number of parts in any overpartition counted by $\op_d^o(n)$ is odd, then the largest singleton of type I must appear, that means the two sets are well-defined. So there are two facts as follows.
\begin{description}
\item[FACT 1] $|\ocP_d^{o1}(n)|+|\ocP_d^{o2}(n)|=\op_n^o(n)$. 

\item[FACT 2] $|\ocP_d^{o1}(n)|=|\ocP_d^{o2}(n)|$. This is because for any given $\la\in \ocP_d^{o1}(n)$, we can find out a corresponding overpartition $\pi\in \ocP_d^{o2}(n)$ by only changing each overlined (resp. non-overlined) singleton of type I to a non-overlined (resp. an overlined) one and vice versa.
\end{description}

Based on these above two facts, we only need to construct the following bijection to claim that the identity \eqref{id:res5-1} holds true.
\begin{align*}
f_0: \cC_0(n)&\ri \ocP_d^{o1}(n)\\
\pi&\mapsto \la
\end{align*}

On the one hand, for any given $\pi\in \cC_0(n)$, recall that $\og(\pi)$ is the greatest overlined part of $\pi$. We obtain $\la\in \ocP_d^{o1}(n)$ from $\pi$ by the following two steps.
\begin{itemize}
\item[(1).] Divide each non-overlined part $\geq 2\og(\pi)+2$ in $\pi$ equally into a pair of type I.

\item[(2).] The remaining parts keep the same as those in $\pi$.
\end{itemize}
We easily see that such $\la$ is well-defined. Firstly, $\og(\pi)$ is the largest singleton of type I in $\la$, and it is overlined. Secondly the number of parts of $\la$ is odd since the number of parts $<\og(\pi)$ is even by the definition of $\cC_0(n)$ and the number of parts $>\og(\pi)$ is even by above operations.

On the other hand, the inverse map $f_0^{-1}$ can be constructed step by step. For a given $\la\in\ocP_d^{o1}(n)$, firstly find out the largest singleton of type I and denote $\og(\pi)$ by it, then we can obtain $\pi\in \cC_0(n)$ from $\la$ by following two steps.
\begin{itemize}
\item[(1).] Merge each pair of type I $>\og(\pi)$ in $\la$ to a non-overlined even part $\geq 2\og(\pi)+2$.

\item[(2).] The remaining parts keep the same as those in $\la$.
 \end{itemize}
 With the same discussion, we can also easily see that such $\pi$ is well-defined. Furthermore, this map $f_0$ is a bijection step by step.

\noindent\textbf{PART II}. We prove the equation \eqref{id:res5-2}. Let $\ocP_{d1}(n)$ be the set of overpartitions of $n$ counted by $\op_d(n)$ where there is no any singleton of type I (that is, all parts are pairs of type I). Then we claim that $|\ocP_{d1}(n)|=p_{ed}(n)$. In fact, we obtain a partition $\la$ counted by $p_{ed}(n)$ from $\pi\in \ocP_{d1}(n)$ by merging each pair of type I to a non-overlined even part, and conversely obtain $\pi\in \ocP_{d1}(n)$ from $\la$ counted by $p_{ed}(n)$ by dividing each part into a pair of type I. Therefore, we can let $\ocP_d^{e1}(n)$ (resp. $\ocP_d^{e2}(n)$) be the set of overpartitions of $n$ counted by $\op_d^e(n)-p_{ed}(n)=|\ocP_d^e(n)-\ocP_{d1}(n)|$ where the largest singleton of type I is overlined (resp. non-overlined). Then there are two facts as follows.
\begin{align}
|\ocP_d^{e1}(n)|+|\ocP_d^{e2}(n)|=\op_n^e(n)-p_{ed}(n),\quad |\ocP_d^{e1}(n)|=|\ocP_d^{e2}(n)|. 
\end{align}

Hence, the same operations as $f_0$ act on $\cC_1(n)$ and $\cP_d^{e1}(n)$, we can obtain the following bijection to finish the proof. Moreover, the identity \eqref{id:res5-2} holds true.
\begin{align*}
f_1: \cC_1(n)&\ri \ocP_d^{e1}(n)\\
\pi&\mapsto \la.
\end{align*}
\end{proof}

\begin{example}
For a given overpartition $\pi=18+16+\overline{7}+6+\overline{5}+3+\overline{1}\in \cC_0(57)$, we firstly confirm the greatest overlined part $\og(\pi)=\overline{7}$. Then we obtain $f_0(\pi)=\overline{9}+9+\overline{8}+8+\overline{7}+6+\overline{5}+3+\overline{1}\in\ocP_d^{o1}(57)$. For another bijection $f_1$, similar examples can be easily given, nothing that only the parity of the number of parts differ.
\end{example}

\begin{proof}[Proof of Theorem \ref{thm:res6}]
We divide the proof into two parts.

\noindent\textbf{PART I}. We prove the identity \eqref{id:res6-1}. Let $\ocP_{d2}(n)$
be the set of overpartitions of $n$ counted by $\op_d(n)$ where there is no any singleton of type II (that is, all parts are pairs of type II). Then we claim that $|\ocP_{d2}(n)=p_{od>1}(n)|$. Indeed, we can obtain a partition $\la$ counted by $p_{od>1}(n)$ from $\pi\in\ocP_{d2}(n)$ by merging each pair of type II to a non-overlined odd part $>1$, and conversely obtain $\pi\in\ocP_{d2}(n)$ from $\la$ counted by $p_{od>1}(n)$ by dividing each part into a pair of type II. Therefore, let $\ocP_d^{e3}(n)$ (resp. $\ocP_d^{e4}(n)$) be the set of overpartitions of $n$ counted by $\op_d^e(n)-p_{od>1}(n)=|\ocP_{d}^e(n)-\ocP_{d1}(n)|$ where the largest singleton of type II is overlined (resp. non-overlined). Note that there is a fact:
\begin{align*}
|\ocP_d^{e3}(n)|+|\ocP_d^{e4}(n)|=\op_d^e(n)-p_{od>1}(n).
\end{align*}
Hence, we will construct a following bijection $h_{0}$ to prove the identity \eqref{id:res6-1}.
\begin{align*}
h_0: \cD_0(n)\cup\cD_0(n-1)&\ri \ocP_d^{e3}(n)\cup\ocP_d^{e4}(n)\\
\pi&\mapsto \la.
\end{align*}
Moreover, if $\pi\in \cD_0(n)$ then $\la=h_0(\pi)\in \ocP_d^{e3}(n)$ and vice versa, and if $\pi\in \cD_0(n-1)$ then $\la=h_0(\pi)\in \ocP_d^{e4}(n)$ and vice versa. Now for any given $\pi\in\cD_0(n)\cup\cD_0(n-1)$, we will divide it into two cases to describe this map $h_0$.
\begin{description}
\item[CASE I] if $\pi\in\cD_0(n)$ with the greatest overlined part $\og(\pi)$, then we obtain $\la\in \ocP_d^{e3}(n)$ by dividing each non-overlined odd parts $\geq 2\og(\pi)+3$ into a pair of type II.

\item[CASE II] if $\pi\in \cD_0(n-1)$ with the greatest overlined part $\og(\pi)$, then we obtain $\la\in\ocP_{d}^{e4}(n)$ by dividing each non-overlined odd parts $\geq 2\og(\pi)+3$ into a pair of type II and changing the overlined part $\og(\pi)$ to the non-overlined part $\og(\pi)+1$. At this time, we have the greatest non-overlined singleton of type II is $g(\la)=\og(\pi)+1$.
\end{description}
On the other hand, we can construct an inverse map $h_0^{-1}$ step by step.
\begin{description}
\item[CASE I'] if $\la\in \ocP_d^{e3}(n)$, then we can obtain $\pi\in\cD_0(n)$ by merging each pair of type II $>\og(\pi)$ (which is the greatest overlined singleton of type II in $\la$) to a non-overlined odd part $\geq 2\og(\pi)+3$.

\item[CASE II'] if $\la\in \ocP_d^{e4}(n)$ with the greatest non-overlined singleton of type II $g(\la)$, then we obtain $\pi\in \cD_0(n-1)$ by merging each pair of type II to a non-overlined odd part $\geq 2g(\la)+1$ and changing the non-overlined part $g(\la)$ to the overlined part $\og(\pi)=g(\la)-1$.
\end{description}

In all of the above cases, the remaining unspecified parts remain unchanged. Now we can easily check each resulting overpartition is well-defined and the map $h_0$ is exactly a bijection.

\noindent\textbf{PART II}. We prove the identity \eqref{id:res6-2}. Let $\ocP_{d3}(n)$ be the set of overpartitions of $n$ counted by $\op_d(n)$ where only a singletons of type II is $1$ and other parts form pairs of type II. Now we can see that $|\ocP_{d3}(n)|=p_{od>1}(n-1)$ since $p_{od>1}(n-1)$ also enumerates the distinct odd partitions of $n$ where the smallest part is $1$. Therefor, we can let $\ocP_{d}^{o3}(n)$ be the set of overpartitions of $n$ counted by $\op_d^o(n)$ where the greatest singleton of type II is overlined, and $\ocP_{d}^{o4}(n)$ be the set of overpartitions of $n$ counted by $\op_d^o(n)$ where the greatest singleton of type II is non-overlined and it should be more than $1$. Similarly we notice the fact
\begin{align}
|\ocP_{d}^{o3}(n)|+|\ocP_{d}^{o4}(n)|=\op_d^o(n)-p_{od>1}(n-1).
\end{align}
Then by the same ways with above the bijection $h_0$, we can obtain the following bijection $h_1$ to finish this proof.
\begin{align*}
h_1: \cD_1(n)\cup\cD_1(n-1)&\ri \ocP_d^{o3}(n)\cup\ocP_d^{o4}(n)\\
\pi&\mapsto \la.
\end{align*}

Furthermore, this bijection tells us that if $\pi\in \cD_1(n)$ then $\la=h_1(\pi)\in \ocP_d^{o3}(n)$ and vice versa, and if $\pi\in \cD_1(n-1)$ then $\la=h_1(\pi)\in \ocP_d^{o4}(n)$ and vice versa.
\end{proof}

\begin{example}
For a given overpartition $\pi^1=23+15+13+\overline{5}+5+2+\overline{1}\in \cD_0(59)$ with $\og(\pi^1)=\overline{5}$, then we obtain the $\la^1=h_0(\pi^1)=12+\overline{11}+8+\overline{7}+7+\overline{6}+\overline{5}+5+2+\overline{1}\in\ocP_d^{e1}(57)$. For a given overpartition $\pi^2=17+15+\overline{6}+6+\overline{2}\in \cD_1(46)$, we firstly confirm the greatest overlined part $\og(\pi^2)=\overline{6}$. Then we obtain the $\la^2=h_1(\pi^2)=9+\overline{8}+8+\overline{7}+7+6+\overline{2}\in \ocP_d^{o4}(47)$ with $g(\la^2)=7$.
\end{example}

\begin{remark}
In this remark, we shall focus on how to obtain Corollaries \ref{cor:3} and \ref{cor:4} from Theorems \ref{thm:res5} and \ref{thm:res6}, respectively.
\begin{itemize}
\item For Corollary \ref{cor:3}, we have
\begin{align*}
C(n)&=C_0(n)+C_1(n)=\frac{\op_d^o(n)+\op_d^e(n)-p_{ed}(n)}{2}=\frac{\op_d(n)-p_{ed}(n)}{2};\\
C'(n)&=C_0(n)-C_1(n)=\frac{\op_d^o(n)-\op_d^e(n)+p_{ed}(n)}{2}=\frac{p_{ed}(n)-\op'_d(n)}{2}.
\end{align*}

\item For Corollary \ref{cor:4}, we have
\begin{align*}
D(n)+D(n-1)&=(D_0(n)+D_0(n-1))+(D_1(n)+D_1(n-1))\\
&=\op_d^e(n)-p_{od>1}(n)+\op_d^o(n)-p_{od>1}(n-1)\\
&=\op_d(n)-p_{od}(n).
\end{align*}
and 
\begin{align*}
D'(n)+D'(n-1)&=(D_1(n)+D_1(n-1))-(D_0(n)+D_0(n-1))\\
&=\op_d^o(n)-p_{od>1}(n-1)-\op_d^e(n)+p_{od>1}(n)\\
&=p_{od>1}(n)-p_{od>1}(n-1)-\op'_d(n).
\end{align*}
\end{itemize}
\end{remark}

%%%%%%%%%%%%%%%%%%%%%%%%%%%%%%%%%%%%%%%%%%%%%%%%%%%
\section{Conclusion}\label{sec:conclusion}

In this paper, we firstly provide two refinements of Corollaries \ref{cor:1} and \ref{cor:2} by the number of parts except for the smallest overlined part, respectively. Then we give the combinatorial proofs of those in Section \ref{sec:pf_c12}. Moreover, for Corollaries \ref{cor:3} and \ref{cor:4} we find out two further theorems to prove them. In fact, we just divide the set of overpartitions with distinct parts into two classes, that is the singletons and pairs of type I, and the singletons and pairs of type II in Section \ref{sec:pf_c34}. This gives us the opportunity to better understand its combinatorial structure, allowing for more meaningful refinements possibly.

%%%%%%%%%%%%%%%%%%%%%%%%%%%%%%%%%%%%%%%%%%%%%%%%%%%
%\section*{Acknowledgement}
%Both authors were partially supported by the National Natural Science Foundation of China grant 12171059 and the Mathematical Research Center of Chongqing University.

%%%%%%%%%%%%%%%%%%%

\end{document}